\documentclass[francais,reqno]{amsart}

\usepackage[francais,english]{babel}
\author{Guillaume Bulteau}
\address{Institut Montpelliérain Alexander Grothendieck (Imag)\\ UMR CNRS 5149 - Université Montpellier 2\\
Case courrier 051\\
34095 Montpellier cedex 5 - France
}
\email{guillaume.bulteau@ac-montpellier.fr}
\title{Géométrie systolique et technique de régularisation}

\usepackage{amssymb,url,xspace}
\usepackage[all]{xy}
\usepackage[latin1]{inputenc}
\usepackage{ae,aecompl}
\usepackage{array}
\usepackage{graphicx}
\graphicspath{{FIGURES/}}
\renewcommand{\baselinestretch}{1.2}
\def\R{\text{\normalfont{I\hspace{-.15em}R}}}
\def\N{\text{\normalfont{I\hspace{-.15em}N}}}
\newcommand {\Z} {\mathbb{Z}}
\renewcommand{\dfrac}{\displaystyle \frac}
\newcommand{\dint}{\displaystyle \int}
\newcommand{\dsum}{\displaystyle \sum}
\newcommand{\norm}[1]{\left|\!\left| #1 \right|\!\right|}
\renewcommand{\mod}[1]{\left| #1 \right|}
\renewcommand{\le}{\leqslant}
\renewcommand{\ge}{\geqslant}
\newcommand {\eps} {\varepsilon}
\newcommand{\sys}{\mbox{\normalfont{\text{syst}}}}
\newcommand {\vol}{\mbox{\normalfont{\text{vol}}}}

\newcommand {\Long}{\mbox{\normalfont{long}}}
\newcommand{\dist}{\normalfont{\text{dist}}}

\newcommand {\id}{\normalfont{\text{Id}}}
\newcommand{\quot}[2]{\raise0.7ex\hbox{$#1$} \!\mathord{\left/
 {\vphantom { {n}}}\right.\kern-\nulldelimiterspace}
\!\lower0.7ex\hbox{$#2$}}
\newcommand {\K} {K(\pi,1)}

\renewcommand{\d}{\normalfont{\text{d}}}
\newcommand{\eq}{\begin{eqnarray*}}
\newcommand{\fineq}{\end{eqnarray*}}
\newcommand{\w}{\widetilde}
\newcommand \set[1]{ \left\{#1\right\}}
\newcommand {\im}{\mbox{\normalfont{Im}}}
\newcommand{\et}{\ \ \text{et}\ \ }

\newcommand {\emb}{\mbox{\normalfont{emb}}}
\newcommand {\aire}{\mbox{\normalfont{aire}}}
\newcommand {\ex}{\mbox{e}^}
\newcommand{\cupp}{\smallsmile}
\newcommand{\capp}{\smallfrown}
\newcommand{\rg}{\normalfont{\text{rg}}}
\newcommand{\trans}{{}^t\!}
\newcommand {\inj}{\mbox{\normalfont{inj}}}
\theoremstyle{plain}
\newtheorem{theo}{Théorème}
\newtheorem{lemm}{Lemme}

\theoremstyle{remark}
\newtheorem{rema}{Remarque}
\theoremstyle{definition}
\newtheorem{defi}{Définition}

\begin{document}

\begin{abstract}
L'objectif de ce texte est de présenter la notion de systole d'une variété riemannienne et de faire un survol de la géométrie systolique. On illustrera aussi une technique fondamentale, appelée technique de régularisation, qui est à la base de plusieurs résultats essentiels de géométrie systolique. Je détaillerai comment cette technique permet d'estimer les nombres de Betti d'une variété asphérique (d'après Gromov), et comment elle permet de relier l'entropie volumique à la systole et au volume systolique d'une variété riemannienne (d'après Sabourau). 
\end{abstract}


\thanks{Je remercie Benoît Kloeckner pour les remarques et suggestions qu'il a pu faire lors de mon passage au séminaire de théorie spectrale et de géométrie de Grenoble, ainsi qu'Ivan Babenko pour les nombreuses discussions sur ce sujet. \\
Ce travail est financé par l'ANR Finsler}
\maketitle

\renewcommand\contentsname{Sommaire}

\tableofcontents
\section{Introduction}
Je vais dans un premier temps présenter la notion de systole d'une variété riemannienne fermée, en m'appuyant sur l'exemple du tore, puis faire un bref survol de certains thèmes de géométrie systolique. 

\subsection{Systole d'une variété}
Considérons un tore plongé dans $\R^3$. On s'intéresse au(x) plus petit(s) lacet(s) non contractile(s) de ce tore. La \textit{systole} du tore est alors la longueur d'un tel lacet (dont je justifierai brièvement l'existence au paragraphe \ref{par-existencesys}), que l'on appellera lacet systolique. 

\begin{figure}[h]
\begin{center}
\includegraphics[width=0.7\linewidth]{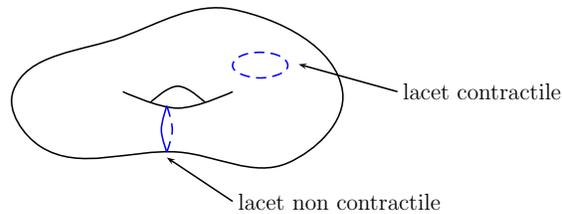}
\caption{Un tore plongé dans $\R^3$.}
\label{fig-torus}
\end{center}
\end{figure}

La figure \ref{fig-torus} laisse penser, qu'à systole fixe, l'aire de ce tore est minorée : si l'on diminue trop l'aire un lacet non contractile plus petit apparaîtra.


Ce phénomène se retrouve clairement pour les tores plats rectangulaires obtenus par identification des côtés opposés d'un rectangle (voir \ref{fig-flattorus}).

\begin{figure}[h]
\begin{center}
\includegraphics[width=0.7\linewidth]{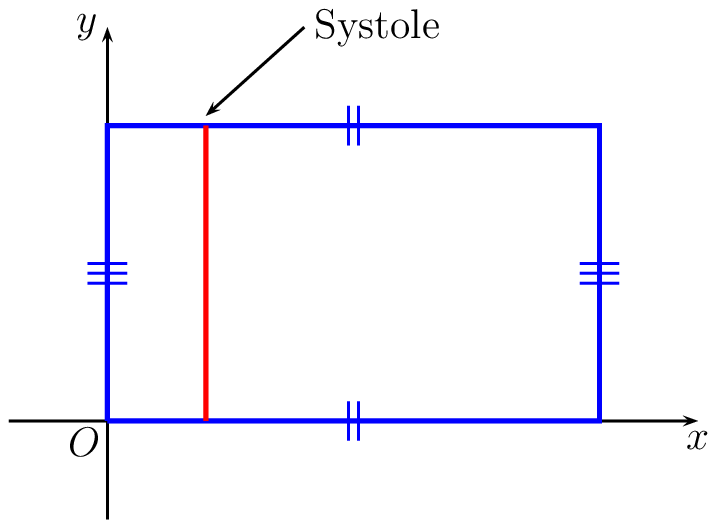}
\caption{La systole  contrôle l'aire d'un tore plat rectangulaire.}
\label{fig-flattorus}
\end{center}
\end{figure}

En fait, on a le résultat suivant, qui est à l'origine de la géométrie systolique.

\begin{theo}[Loewner - 1949]
Soit $M$ un tore de dimension 2. Pour toute métrique riemannienne $g$ sur $M$, on a :
$$\aire(M,g)\ge \dfrac {\sqrt 3}{2}\:\sys(M,g),$$
où $\sys(M,g)$ désigne la longueur du plus petit lacet non contractile de $M$.

Il y a égalité si et seulement si $(M,g)$ est le tore plat hexagonal, obtenu en identifiant les côtés opposés du losange de la figure \ref{fig-hextorus}.
\end{theo}
\begin{figure}[h]
\begin{center}
\includegraphics[width=0.7\linewidth]{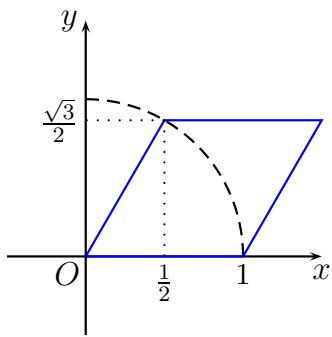}
\caption{Le tore hexagonal.}
\label{fig-hextorus}
\end{center}
\end{figure}

On peut trouver une preuve de ce résultat, qui repose sur le théorème de la représentation conforme, dans \cite{Berger-Loewner} ou \cite{Katz-SG}.

\subsection{Existence de la systole} 
\label{par-existencesys}On va justifier, dans ce paragraphe, l'existence de lacets systoliques. On se place, de manière plus générale, sur une variété riemannienne fermée $(M,g)$, non simplement connexe, de dimension $n\ge 1$.

Un lacet dans $(M,g)$ sera une application $c:[0,1]\to \R$, continue, telle que $c(0)=c(1)$. La systole\index{Systole} de $(M,g)$, notée $\sys(M,g)$, est la longueur du plus petit lacet non contractile dans $M$.

Notons $\Omega M$ l'ensemble des lacets dans $M$. On munit $\Omega M$ de la distance $\delta$ définie par :
$$\delta(\alpha,\beta)=\underset{t\in [0,1]}{\sup} \dist(\alpha(t),\beta(t)).$$
 Une classe \textit{d'homotopie libre} d'un lacet $\alpha$ est l'ensemble des lacets dans $M$ homotopes à $\alpha$. Ces classes d'homotopies libres sont fermées dans $(\Omega M,\delta)$. En effet deux lacets suffisamment proches dans $\Omega M$ sont homotopes; ainsi les classes d'homotopies libres sont ouvertes, donc fermées, dans $(\Omega M,\delta)$.

Maintenant, en utilisant le théorème d'Ascoli, on va prouver qu'il existe, dans chaque classe d'homotopie libre, un lacet de longueur minimale. Soit $\alpha_0$ un lacet dans $M$. On considère la borne inférieure $\ell$ de l'ensemble des longueurs des lacets de $M$ homotopes à $\alpha_0$. On considère une suite minimisante $(\gamma_n)$ de lacets de $M$ homotopes à $\alpha_0$ tels que :
$$\Long(\gamma_n)\le \Long(\alpha_0)+\dfrac 1{2^n}.$$
La longueur de ces courbes est uniformément bornée. Le théorème d'Ascoli assure, quitte à prendre une suite extraite, que la suite $(\gamma_n)$ converge vers un lacet $\gamma$ de $M$. Ce lacet $\gamma$ est homotope à $\alpha_0$, puisque les classes d'homotopies libres sont fermées. Mais la fonction  longueur est semi-continue inférieurement (voir \cite{BBI}, page 34), donc :
$$\Long (\gamma)\le \liminf \Long(\gamma_n)=\ell.$$
Il en résulte que $\Long (\gamma)=\ell$.

Enfin, pour tout $L>0$, le nombre de classes d'homotopie libre représentées par un lacet $\alpha$ de longueur inférieure à $L$ est fini. Pour montrer cela, on raisonne par l'absurde. On suppose  qu'il existe une suite $(\gamma_n)$ dans $\Omega M$ telle que :

\begin{itemize}
	\item [$\bullet$] $\Long(\gamma_n)\le L$;

\item [$\bullet$] Pour $p\neq q$  dans $\N$, les lacets $\gamma_p$ et $\gamma_q$ ne sont pas homotopes.
\end{itemize}
Toujours par le théorème d'Ascoli, il existe une extraction $\varphi$ telle que la suite $(\gamma_{\varphi(n)})$ converge vers un lacet $\gamma$. Mais la classe d'homotopie libre de $\gamma$ est ouverte dans $\Omega M$ : pour $n$ assez grand tous les lacets $\gamma_{\varphi(n)}$ seront homotopes à $\gamma$, ce qui est contradictoire.  

On peut ainsi considérer la plus petite longueur des lacets non contractiles de $M$. On peut montrer sans trop de peine que les lacets systolique de $M$, c'est à dire les lacets qui réalisent cette longueur, sont des géodésiques.  

\subsection{Le résultat de Burago et Hebda} Le contrôle de l'aire du tore par sa systole s'étend aux surfaces fermées de genre supérieur à 1, de manière assez élémentaire (voir \cite{BZ} page 43 ou \cite{Hebda}).

\begin{theo}[Burago/Hebda - 1980]
Toute surface fermée  $(M,g)$ de genre $h\ge 1$ contient une courbe fermée non contractile de longueur $\ell$ telle que $\aire(M,g)\ge \dfrac 12\:\ell$.
\end{theo}

\begin{proof}[Démonstration]
Soit $c$ un lacet systolique de $M$, que l'on suppose paramétré par la longueur d'arc, et soit $m$ un point de $c$. On suppose que $c(0)=m$. On note $\ell=\Long(c)$. Comme $c$ est une géodésique, pour tout $t\in [-\tfrac{\ell}2,\tfrac \ell2]$, on a $\dist(m,c(t))=t$. Si $B=B(m,\tfrac \ell 2)$ désigne la boule de centre $m$ est de rayon $\tfrac \ell2$, montrons que l'on a  :

\begin{equation}
\label{eq-Hebda}
	\aire \big( B(m,\tfrac \ell2)\big)\ge \dfrac{\ell^2}2
\end{equation}

Pour tout $r\in [0,\tfrac \ell2[$, on considère la boule fermée $\overline B(m,r)$, qui est, de par la définition de $\ell$, contractile. Regardons maintenant son bord $S(m,r)$. On a :
$$S(m,r)=\gamma_1\cup \gamma_2,$$
où $\gamma_1$ et $\gamma_2$ sont deux courbes qui relient les points $c(r)$ et $c(-r)$ (voir figure \ref{fig-Hebda}). 

\begin{figure}[h]
\begin{center}
\includegraphics[width=0.7\linewidth]{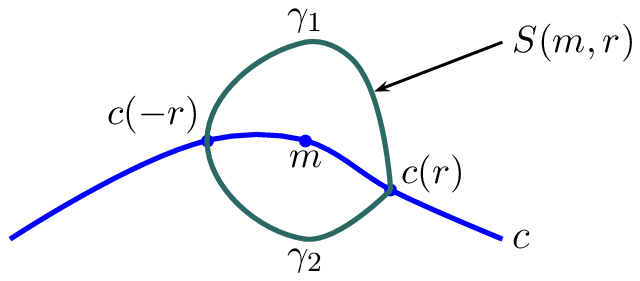}
\caption{La boule $\overline{B}(m,r)$.}
\label{fig-Hebda}
\end{center}
\end{figure}
Si par exemple $\Long (\gamma_1)<2r$, en recollant $\gamma_1$ et $c$, on obtient un lacet non contractile  
de longueur strictement inférieure à $\ell$, ce qui est absurde. On a donc $\Long(\gamma_1)\ge 2r$ et, de même, $\Long(\gamma_2)\ge 2r$.

Ainsi :$$\aire (B)=\dint_0^{\frac\ell 2}\Long S(m,r)\d r\ge4\dint_0^{\frac\ell 2} r\d r=\dfrac{\ell^2}{2}.$$
Le théorème est donc prouvé, puisque $\aire(M,g)\ge \aire (B)$.
\end{proof}
\section{Un survol de la géométrie systolique}


\subsection{La géométrie systolique}
Soit  $(M,g)$ une variété riemannienne fermée, non simplement connexe, de dimension $n\ge 1$. Nous avons vu que la \textit{systole} de $(M,g)$, notée $\sys(M,g)$, est la longueur du plus petit lacet non contractile dans $M$. 
Le \textit{volume systolique}\index{Volume systolique} de $M$ est alors :
$$\sigma(M)=\underset{g}{\inf}\dfrac{\vol(M,g)}{\sys(M,g)^n},$$
où $g$ décrit l'ensemble des métriques riemanniennes lisses sur $M$.

Le thème principal de la géométrie systolique est l'étude du volume systolique. La question suivante est centrale :

\begin{quote}
	\textit{Sous quelles conditions topologiques a-t-on $\sigma(M)>0$ et quelle est sa valeur exacte ?}\\
\end{quote}

Une réponse satisfaisante a été donnée par Gromov à la première partie de la question en utilisant la notion de \textit{variété essentielle}. 

Soit $M$ une variété compacte de groupe fondamental $\pi$. Si $M$ est orientable on note $A=\Z$ et, dans le cas contraire, $A=\Z_2$. Désignons par $[M]$ la classe fondamentale de $M$ dans le groupe d'homologie $H_n(M;A)$. Il existe  une application 
$$f:M\to K(\pi,1),$$
unique à homotopie près, où $K(\pi,1)$ est le  complexe d'Eilenberg-McLane, telle que le morphisme induit par $f$ sur $\pi$ soit égal à l'identité. La variété $M$ est alors dite \textit{essentielle}\index{Variété essentielle} lorsque l'image de $[M]$ par le  morphisme induit en homologie 
$$f_*:H_n(M;A)\to H_n(\K;A)$$ est une classe d'homologie non nulle dans $H_n(\K;A)$. Gromov a démontré (voir \cite{Gromov-FRM}, page 3) le résultat fondamental suivant.\\

\begin{theo}[Gromov]
\label{Gromov}
\index{Inégalité systolique de Gromov}
Il existe une constante $C_n>0$ telle que pour toute variété compacte essentielle $M$ de dimension $n$ on ait :
$$\sigma(M)\ge C_n.$$
\end{theo}

Une preuve détaillée de ce résultat se trouve dans \cite{Guth-Note}. La catégorie des variétés essentielles recouvre notamment :
\begin{enumerate}
	\item Les variétés asphériques i.e. les variétés dont les groupes d'homotopie d'ordre supérieur à 2 sont triviaux, en particulier les variétés à courbures négatives.
	\item Les espaces projectifs $\R P^n$. 
\end{enumerate}

Babenko a démontré, dans le cas où $M$ est orientable, qu'il est nécessaire que $M$ soit essentielle pour avoir $\sigma(M)>0$ (voir \cite{Babenko-AISM}). Dans le même article, il est montré que $\sigma(M)$ est un invariant du type d'homotopie de $M$ : deux variétés ayant même type d'homotopie ont le même volume systolique.\\

Babenko a aussi démontré que, pour une variété essentielle orientable $M$ de groupe fondamental $\pi$, le volume systolique $\sigma(M)$ ne dépend que de l'image $h$ dans $H_n(\K;\Z)$ de la classe fondamentale $[M]$ par l'application induite en homologie (voir \cite{Babenko-TSU}).\\ 

La réponse à la deuxième partie de la question, c'est à dire la détermination de la valeur exacte du volume systolique, est un problème difficile. Le volume systolique des variétés essentielles n'est connu que dans trois cas : 
\begin{itemize}
	\item le tore de dimension 2, pour lequel $\sigma(T^2)=\dfrac {\sqrt 3}{2}$ (Loewner, non publié mais exposé dans  \cite{Berger-Loewner});
	 \item le plan projectif $\sigma(\R P^2)=\dfrac {2}{\pi}$(Pu, voir \cite{Pu} et \cite{Berger-Pu});
	 \item la bouteille de Klein $\sigma(K^2)=\dfrac{2\sqrt 2}{\pi}$ (Bavard, voir \cite{Bavard-Klein} et \cite{Sakai-Klein}).\\
\end{itemize}
 
D'autres questions concernant le volume systolique des variétés essentielles se posent :
\begin{quote}
	
\begin{itemize}
	\item \textit{La borne inférieure de la définition du volume systolique est-elle atteinte ?}
	\item \textit{Dans l'affirmative, quelles sont les métriques optimales sur $M$ ?}\\
\end{itemize}
\end{quote}

Il s'agit encore de questions difficiles dont les réponses sont très partielles. La métrique optimale pour la bouteille de Klein présente certaines singularités. Gromov a démontré que le volume systolique des surfaces compactes est atteint pour des métriques riemanniennes généralisées en un certain sens (voir \cite{Gromov-FRM}, page 63). \\

Pour les surfaces orientables $S_h$ de genre $h$, on dispose d'une estimée inférieure du volume systolique (voir \cite{Gromov-FRM} et \cite{Kodani}) qui est :
$$\sigma(S_h)\ge \dfrac{4\sqrt h+27}{64}.$$

On connaît aussi partiellement le comportement asymptotique de $\sigma(S_h)$ (voir \cite{Sabourau-Katz-Ent} et \cite{Buser-Sarnak}) :
$$\pi\le\underset{h\to +\infty}{\limsup} \ \sigma(S_h)\times \dfrac{\log^2 h} {h}\le  \frac{9\pi}{4}.$$

On peut définir la notion de volume systolique pour des objets plus généraux que des variétés. Par exemple si $X$ est un espace de longueur compact, la notion de systole est parfaitement définie (voir \cite{Reviron}), et on peut alors utiliser le volume de Hausdorff pour définir le volume systolique (voir \cite{Gromov-MS}).\\

\subsection{Volume systolique des polyèdres}
Si maintenant $X$ est un polyèdre (espace topologique muni d'une triangulation) de groupe fondamental $\pi$, il existe une unique application à homotopie près 
$$f:X\to \K$$
telle que le morphisme de $\pi$ induit par $f$ soit égal à l'identité : c'est l'application classifiante\index{Application classifiante}. Un polyèdre compact $X$ est dit $n$-essentiel lorsque l'application classifiante ne peut être déformée dans le squelette de dimension $n-1$ de $\K$. 

On peut munir tout polyèdre de dimension $n$ de métriques polyédrales. Une \textit{métrique polyédrale} sur $X$ est  une famille de métrique riemannienne $(g_\sigma)_{\sigma\in \mathfrak S}$, où $\mathfrak S$ est l'ensemble des simplexes de $X$, qui vérifie :
\begin{quote}
\begin{itemize}
	\item Chaque $g_\sigma$ est une métrique riemannienne lisse à l'intérieur de $\sigma$ ;
	\item Dès que $\sigma_1$ et $\sigma_2$ sont dans $\mathfrak S$, on a l'égalité $$g_{\sigma_1}\mid_{ \sigma_1\cap \sigma_2}=g_{\sigma_2}\mid_{\sigma_1\cap \sigma_2}.$$
\end{itemize}
\end{quote}
On désignera par une seule lettre $g$ la famille $(g_\sigma)_{\sigma\in \Sigma}$ et on dira que $g$ est une métrique riemannienne lisse par morceaux (ou polyédrale) sur $X$.

Ler polyèdre $X$ est alors est un espace de longueur pour la distance induite par cette famille de métriques. La systole et le volume de $X$ sont correctement définis. Le \textit{volume systolique} de $X$ est  :
$$\sigma(X)=\underset{g}{\inf}\dfrac{\vol(X,g) }{\sys(X,g)^n},$$
où $g$ décrit l'ensemble des métriques polyédrales sur $X$.\\

Le théorème \ref{Gromov} reste valable pour les polyèdres $n$-essentiels de dimension $n$ (\cite[Appendice 2]{Gromov-FRM}).

\subsection{Aire systolique d'un groupe de présentation finie}
Le volume systolique d'un complexe simplicial permet de définir l'aire systolique\index{Aire systolique d'un groupe} d'un groupe $G$ de présentation finie, qui est :
$$\sigma(G)=\underset{P}{\inf}\: {\sigma(P)},$$
où  $P$ décrit l'ensemble des complexes simpliciaux de dimension inférieure ou égale à 2 et  de groupe fondamental $G$. 

Pour tout groupe non libre de présentation finie,  on a (voir \cite{Sabourau-Rudyak2}) : $$\sigma(G)\ge \dfrac{\pi}{16}.$$

Lorsque $G$ est un groupe de présentation finie non trivial, sans facteur libre isomorphe à $\Z$, alors :
$$\sigma(G)\ge C\dfrac{b_1(G)}{\log^2(b_1(G)+2)},$$
où $C>0$ est une constante universelle, et $b_1(G)$ est le premier nombre de Betti de $G$ (\cite{Balacheff-Parlier-Sabourau}).

L'aire systolique d'un groupe de présentation finie $G$ est intimement liée à la notion de complexité simpliciale d'un groupe, qui est le nombre minimal de simplexes de dimension 2 d'un complexe simplicial deux dimensionnel de groupe fondamental $G$ (voir \cite{BBB}).

\subsection{Systole relative et volume systolique relatif} Soit $\pi$ un groupe de présentation finie. Considérons un complexe simplicial $X$ de dimension $n$, muni d'une métrique polyédrale $g$ tel qu'il existe une application continue $f:X\to \K$. 

La systole relative\index{Systole relative} de $X$, notée ici $\sys(X,f,g)$, est la longueur du plus petit lacet $c$ dans $X$ tel que le lacet $f\circ c$ soit non contractile dans $\K$. On peut alors définir le volume systolique relatif\index{Volume systolique relatif} de $V$ :
$$\sigma(X,f,g)=\dfrac{\vol(X,g)}{\sys(V,f,g)^n}.$$ 

Pour un polyèdre de dimension 2, lorsque l'application $f$ ne se déforme pas dans le squelette de dimension 1 de $\K$, on a (voir \cite{Sabourau-Katz2}) :
$$\dfrac{\aire(X,g)}{\sys(V,f,g)^2}\ge \dfrac 18.$$\\

\subsection{Volume systolique d'une classe d'homologie}
Soit $\pi$ un groupe de présentation finie. On considère une classe d'homologie non nulle $h$ dans $H_n(\K;\Z)$. Cette classe peut être représentée par des pseudo-variétés $V$, compactes et orientables, de dimension $n$. 
Rappelons qu'une \textit{pseudo-variété}\index{Pseudo-variété} de dimension $n$ est un complexe simplicial fini $K$ tel que :

\begin{itemize}
	\item $\dim K=n$;
	\item Chaque simplexe dans $K$ est face d'un simplexe de dimension $n$ (homogénéité de la dimension);
	\item Chaque simplexe de dimension $n-1$ est face d'exactement deux simplexes de dimension $n$ (pas de bifurcation);
	\item Dès que $\sigma$ et $\tau$ sont deux simplexes distincts de dimension $n$ dans $K$, il existe une suite $\sigma_1=\sigma,\ldots,\sigma_p=\tau$ de simplexes de dimension $n$ dans $K$ tels que pour tout $i$ dans $\set{1,p-1}$ les simplexes $\sigma_i$ et $\sigma_{i+1}$ aient une face de dimension $n-1$ commune (forte connexité).\\
\end{itemize}

Il est équivalent de dire qu'un complexe simplicial $K$ est une pseudo-variété de dimension $n$ lorsque $K$ est de dimension $n$ et lorsqu'il existe un sous-complexe $\Sigma\subset K$ vérifiant les trois propriétés suivantes :

\begin{itemize}
	\item $\dim \Sigma \le n-2$;
	\item $K\setminus \Sigma$ est une variété topologique de dimension $n$ dense dans $K$;
	\item L'espace $K\setminus \Sigma$ est connexe.\\
\end{itemize}

Lorsque $V$ est une pseudo-variété et $f:V\to \K$, on dit que $(V,f)$ représente $h$ lorsque $f_*[V]=h$, où 
$$f_*:H_n(V;\Z)\to H_n(\K;\Z)$$
est le morphisme induit en homologie, et $[V]$ la classe fondamentale de $V$. La représentation de $h$ par le couple $(V,f)$ est dite \textit{normale} lorsque le morphisme $f_*:\pi_1(V)\to \pi$ induit par $f$ est un épimorphisme. \index{Représentation normale}\\

On peut munir chacune de ces pseudo-variétés d'une métrique polyédrale $g$. Le triplet $(V,f,g)$ s'appelle un \textit{cycle géométrique} représentant $h$. On définit alors le \textit{volume systolique}  de $h$ par 
$$\sigma(h)=\underset{(V,f,g)}{\inf}\sigma(V,f,g),$$
 où $(V,f,g)$ décrit l'ensemble des cycles géométriques qui représentent la classe d'homologie $h$. \\
 
D'après \cite{Gromov-FRM}, pour toute classe d'homologie non nulle $h$ dans $H_n(\pi;\Z)$, on a $\sigma(h)>0$. 

Lorsque $(V,f)$ est une représentation normale de $h$, on dit qu'elle est \textit{admissible}\index{Représentation admissible} lorsque tout élément de $\pi_1(V)$ peut être représenté par un lacet ne passant pas par le lieu singulier de $V$. Dans le cas où $(V,f)$ est une représentation normale admissible de $h$ on a :
$$\sigma(h)=\underset{g}{\inf}\: \sigma(V,f,g),$$
où $g$ décrit l'ensemble des métriques polyédrales sur $V$ (voir \cite{Babenko-Balacheff-DVS2}). Soulignons d'ailleurs que $h$ admet toujours une représentation normale admissible. \\

Cette présentation de la géométrie systolique est forcément incomplète. On pourra la compléter en consultant les articles \cite{Berger-Systole}, \cite{Croke-Katz} et \cite{Gromov-Actes}, ainsi que le livre \cite{Berger-PV}. Le livre \cite{Katz-SG} est un état des lieux en 2007. Deux autres références présentent des résultats fondamentaux de cette géométrie : \cite{Gromov-FRM}, qui peut en être considéré comme le point de départ, et \cite{Gromov-MS}, qui est une version augmentée de \cite{Gromov-SM}.\\ 

Donnons enfin comme références, l'article \cite{Guth-MSG}, qui relie des idées essentielles de la géométrie systolique à d'autres branches des mathématiques, ainsi que l'article \cite{Alvarez-Balacheff}, qui fait le lien entre la géométrie systolique et la géométrie de contact.

\section{La technique de régularisation}

\subsection{Cycle géométriques réguliers} 
Soient $\pi$ un groupe de présentation finie et $h$ une classe d'homologie dans $H_n(\pi;\Z)$. Rappelons qu'un cycle géométrique représentant $h$ est un triplet $(V,f,g)$ où 
\begin{itemize}
	\item $V$ est une pseudo-variété de dimension $n$, compactes et orientables ;
	\item $f:V\to \K$ est une application continue telle, qu'au niveau de l'homologie, $f_*[V]=h$, où $[V]$ désigne la classe fondamentale de $V$. 
	\item $g$ est une métrique polyédrale sur $V$.
\end{itemize}
Les objets que l'on va utiliser dans la suite sont des cycles géométriques particuliers qui représentent une classe d'homologie entière non nulle dans $H_n(\pi;\Z)$, dans lesquelles on dispose d'un contrôle du volume des petites boules. Ils sont définis par le théorème suivant.

\begin{theo}[Gromov]
\label{theo-reg}
Soient $\pi$ un groupe de présentation finie et $h$ une classe d'homologie entière non nulle dans $H_n(\pi;\Z)$. Pour tout $\eps$ dans $]0,\tfrac12\sys(V,f,g)[$, il existe un cycle géométrique $(V,f,g)$ représentant $h$ tel que :
\begin{enumerate}
	\item $\sigma(V,f,g)\le \sigma(h)+\eps$.
	\item Pour $R\in [\eps,\tfrac 12\:\sys(V,f,g)]$, les boules $B(R)$ de rayon $R$ dans  $V$ vérifient :
\begin{equation}
\label{eq-volboule}
			\vol (B(R))\ge A_n R^n,
\end{equation}
pour une certaine constante universelle $A_n$, qui ne dépend que de la dimension de $h$.
\end{enumerate}
Un tel cycle géométrique est dit $\eps$-régulier\index{Cycle géométrique régulier}.
\end{theo}

On peut comparer ce résultat avec l'inégalité \ref{eq-Hebda}, page \pageref{eq-Hebda}. Ce théorème, dont une démonstration exhaustive se trouve dans \cite{Bulteau}, est à la base de plusieurs résultats importants de géométrie systolique. La \textit{technique de régularisation} consiste à utiliser ces cycles réguliers pour obtenir des informations topologiques sur une classe $h\in H_n(\pi;\Z)$. Elle a permis notamment de relier le volume systolique d'une variété fermée asphérique à d'autres invariants homotopiques de cette variété. 

Je présenterai dans le paragraphe 4 comment obtenir, d'après Gromov, une majoration des nombres de Betti d'une variété fermée asphérique en fonction de son volume systolique. Puis, je détaillerai, dans le paragraphe 5, la preuve d'un théorème de Sabourau, reliant le volume systolique  et l'entropie volumique d'un cycle géométrique régulier, afin de souligner les idées essentielles de géométrie systolique qui y interviennent.

\subsection{Premières illustrations de la technique de régularisation}
Je vais donner ici quelques illustrations  du théorème \ref{theo-reg}. Commençons par une situation élémentaire  qui montre comment l'inégalité (\ref{eq-volboule}) sur le volume des boules permet de préciser la topologie des cycles géométriques réguliers. 

Soit $(V,f,g)$ un cycle $\eps$-régulier qui représente une classe d'homologie non triviale dans $H_n(\pi;\Z)$. Considérons un système maximal $B_1,\ldots,B_N$ de boules ouvertes disjointes de $V$ de rayon $R_0=\tfrac 1{2}\:\sys(V,f,g)$. Les boules concentriques $2B_1,\ldots,2B_N$ de rayon $2R_0$ recouvrent $V$. Appelons $\mathcal N$ le nerf de ce recouvrement\index{Nerf d'un recouvrement}. Il s'agit du complexe simplicial construit de la manière suivante :

\begin{quote}
\begin{itemize}
	\item Les sommets $p_1,\ldots,p_N$ de $\mathcal N$ sont identifiés avec les boules du recouvrement ;
	\item Deux sommets $p_i$ et $p_j$ sont reliés par une arête lorsque $2B_i\cap 2B_j\neq \emptyset$;
	\item Pour $p\ge 2$ entier, les sommets $p_{i_0},\ldots ,p_{i_p}$ forment un simplexe de dimension $p$ de $\mathcal N$ lorsque :
	$$2B_{i_0}\cap\cdots\cap 2B_{i_p}\neq \emptyset.$$
\end{itemize}
\end{quote}

On peut alors borner le nombre $N_k$ de simplexes de dimension $k$ de $\mathcal N$ en fonction du volume systolique $\sigma(h)$ de la classe $h$. Par exemple, on a $$\vol(V,g)\ge \dsum_{i=1}^{N_0}\vol (B_i)\ge N_0 A_n R_0^n,$$
ce qui permet de borner le nombre de sommets $N_0$ de $\mathcal N$ en fonction de $\sigma(h)$.\\


La première conséquence importante du théorème \ref{theo-reg} concerne le volume systolique d'une classe d'homologie non nulle dans $H_n(\pi;\Z)$, lorsque $\pi$ est un groupe de présentation finie.

\begin{theo}
\label{theo-sigmah}
Soient $\pi$ un groupe de présentation finie et $n\ge 1$. Il existe une constante $C_n>0$, qui ne dépend que de $n$, telle que pour toute classe d'homologie non nulle $h$ dans $H_n(\pi;\Z)$ on ait :
$$\sigma(h)\ge C_n.$$ 
\end{theo}

\begin{proof}[Démonstration]
Soit $h\in H_n(\pi;\Z)$. Fixons  provisoirement $\eps>0$. Il existe alors, selon le théorème A, un cycle géométrique $\eps$-régulier $(V,f,g)$ qui représente la classe $h$. Soit $v\in V$. On a alors :
$$\vol(V,g)\ge \vol\big(B(v,\tfrac 12\sys(V,f,g)\big)\ge \dfrac{A_n}{2^n}\:\sys(V,f,g)^n.$$
Il en résulte que $\sigma (V,f,g)\ge \dfrac{A_n}{2^n}$. Ainsi :
$$\sigma(h)+\eps\ge \dfrac{A_n}{2^n}.$$
En faisant tendre $\eps$ vers $0$, on obtient le résultat souhaité, avec $C_n=\dfrac{A_n}{2^n}$.
\end{proof}

On peut alors en déduire une preuve de l'inégalité systolique de Gromov (théorème \ref{Gromov}). 
\begin{quote}
Soit $M$ une variété essentielle orientable de dimension $n$, de groupe fondamental $\pi$. Il existe une application $f:M\to \K$ telle que $f_*[M]=h\neq 0$, où 
$$f_*:H_n(M;\Z)\to H_n(\pi;\Z)$$
est le morphisme induit en homologie. D'après le théorème \ref{theo-sigmah}, on a $\sigma(h)\ge C_n$, pour une constante universelle $C_n>0$ qui ne dépend que de $n$. Mais $(M,f)$ est une représentation normalisée admissible de $h$ : selon \cite{Babenko-Balacheff-DVS2}, on a $\sigma(M)=\sigma(h)$.\hfill $\square$
\end{quote}

Gromov utilise aussi le théorème \ref{theo-reg} afin de relier des propriétés topologiques de $h$ au volume systolique. Plus précisément, ces cycles réguliers permettent à Gromov d'obtenir des inégalités entre le volume systolique et deux invariants topologiques importants de la classe $h$, qui sont :

\begin{itemize}
	\item La hauteur simpliciale $h_s(h)$ de $h\in H_m(\pi;\Z)$, qui est le nombre minimal de simplexes de toute dimension d'un cycle géométrique qui représente $h$;
	\item Le volume simplicial\index{Volume simplicial} $\norm{h}_\Delta$, défini comme l'infimum des sommes $\dsum \mod{r_i}$ sur toutes les représentations de $h$ par des cycles singuliers réels $\dsum r_i\sigma_i$.\\
	\end{itemize}

Gromov a notamment obtenu les résultats suivants (voir \cite{Gromov-FRM}, théorème 6.4.C'' et théorème 6.4.D' et \cite{Gromov-Actes}, paragraphe 3.C.3).\\

\begin{theo}[Gromov]
\label{theoGromovIllus}
Soit $\pi$ un groupe de présentation finie, $h\in H_m(\pi;\Z)$, une classe d'homologie non nulle de dimension $m\ge 2$.
\begin{enumerate}
	\item Il existe deux constantes positives $C_m$ et $C_m'$, qui ne dépendent que de $m$, telles que :
	$$\sigma(h)\ge C_m\dfrac{h_s(h)}{\exp(C_m'\sqrt {\ln h_s(h)})}.$$
\item	Il existe une constante positive $C_m''$ qui ne dépend que de la dimension $m$ telle que :
$$\sigma(h)\ge C_m''\dfrac{\norm{h}_\Delta}{(\ln (2+\norm{h}_\Delta))^m}.$$
\end{enumerate}
\end{theo}


\renewcommand{\baselinestretch}{1.2}


\section{Nombres de Betti d'une variété asphérique}
\subsection{Introduction}

Soient $\pi$ un groupe de présentation finie et $h\in H_n(\pi;\Z)$. La technique de régularisation permet de relier le nombre de simplexes d'un polyèdre particulier au volume systolique de $h$. Plus précisément, on va démontrer le théorème suivant (voir \cite{Gromov-FRM}, page 72).

\begin{theo}[Gromov]
\label{theopolyedre}
Soient $\pi$ un groupe de présentation finie, $h$ une classe d'homologie non nulle dans $H_n(\pi;\Z)$. Il existe un polyèdre $\mathcal P$ et une application $g:\mathcal P\to \K$ tels que :
\begin{itemize}
\item $g$ induit un épimorphisme de groupes $g_*:\pi_1(\mathcal P)\to \pi_1(M)$;
	\item le nombre $N_k$ de simplexes de dimension $k$ de $\mathcal P$ vérifie :
$$
		N_k\le  C_n\sigma(h)\exp\big(C_n'\sqrt{\log \sigma(h)}\big),
$$
où $C_n$ et $C_n'$ sont deux constantes universelles qui ne dépendent que de $n$;
	\item l'application induite en homologie par $g$ envoie une certaine classe $h'$ de $H_n(\mathcal P;\Z)$ sur la classe $h$.
\end{itemize}
\end{theo}

Dans tout ce paragraphe, $h\in H_n(\pi;\Z)$ et $(V,f,\mathcal G)$ désigne un cycle géométrique $\eps$-régulier qui représente $h$, où $\eps\in ]0,\tfrac12\:\sys(V,f,\mathcal G)]$. On a donc (voir théorème \ref{theo-reg}, page \pageref{theo-reg}) $\sigma(V,f,\mathcal G)\le (1+\eps)\sigma(h)$ et, pour tout $R$ dans $[\eps,\tfrac12 \:\sys(V,f,\mathcal G)]$, les boules $B(R)$ de rayon $R$ dans $V$ vérifient :
$$\vol(B(R))\ge A_nR^n.$$

\subsection{Nerf associé à un recouvrement particulier de $V$}
On considère un recouvrement $\mathcal O$ de $V$ par des boules ouvertes $O_1=B(v_1,R_1),\ldots,O_N=B(v_N,R_N)$ telles que, pour tout $i\in \set{1,\ldots,N}$, on ait :
$$R_i< \dfrac 16\sys(V,f,\mathcal G).$$

On peut réaliser géométriquement $\mathcal P$ dans $\R^N$ de la manière suivante :
 
  \begin{enumerate}
	\item Pour $i\in \set{1,\ldots,N}$, le sommet $p_i$ qui correspond à $O_i$, est le point de $\R^N$ dont toutes les coordonnées sont nulles, sauf la $i$-ème qui est égale à $1$.
	\item Les sommets $p_1,\ldots,p_{m+1}$ forment un simplexe euclidien de dimension $m$ lorsque :
		$$O_1\cap O_2\cap \cdots\cap O_{m+1}\neq \emptyset.$$
\end{enumerate}

A ce recouvrement, on associe une partition de l'unité définie de la manière suivante. On part d'une famille $(\psi_i)_{1\le i\le N}$ de fonctions $\psi_i: V\to \R$, continues, telle que, pour tout $i\in \set{1,\ldots,N}$, la fonction $\psi_i$ soit strictement positive sur $O_i$ et nulle hors de $O_i$. Pour $i$ dans $\set{1,\ldots,N}$ on pose :
$$\phi_i=\dfrac{\psi_i}{\dsum_{j=1}^N\psi_j}.$$
Chacune des fonctions $\phi_i$ est continue, à valeurs dans $[0,1]$, strictement positive sur $O_i$, et on a :
$$\dsum_{i=1}^N \phi_i=1.$$

A cette partition de l'unité, on peut associer une fonction $p:V\to \mathcal P$
qui, à un point $v$ de $V$, associe le point de coordonnées $(\phi_1(v),\ldots,\phi_N(v))$ de $\R^N$. Elle est bien à valeurs dans $\mathcal P$. On peut aussi noter (voir par exemple \cite{Dugundji}, page 172) que $p^{-1}(\text{star}(p_i))\subset O_i$ pour tout $i\in \set{1,\ldots,N}$, où $\text{star}(p_i)$ désigne l'union de l'intérieur de tous les simplexes  de $P$ qui contiennent $p_i$.

\begin{rema}
Si on considère une autre partition de l'unité $\dsum_{i=1}^N\overline{\varphi}_i=1$ associée au recouvrement $\mathcal O=(O_i)_{1\le i\le N}$ de $V$, la fonction associée $\overline p$ est homotope à $p$. En effet, pour $t$ dans $[0,1]$ et $i\in \set{1,\ldots,N}$, on pose : $H_i(t,.)=(1-t)\varphi_i+t\overline{\varphi}_i$.
Pour tout $t\in [0,1]$, on obtient alors une partition de l'unité $\dsum_{i=1}^N H_i(t,.)$, d'où l'application $H(t,.)=(1-t)p+t\overline p$,
qui permet de réaliser une homotopie de $p$ à $\overline p$.
\end{rema}

Il existe alors une application naturelle de $\mathcal P$ dans $\K$. 

\begin{lemm}
\label{lemhomotopie}
Avec les notations précédentes, il existe une application continue $g:\mathcal P\to \K$ qui induit un épimorphisme au niveau des groupes fondamentaux telle que $g\circ p$ soit homotope à $f$, i.e. telle qu'à homotopie près, le diagramme suivant commute.
$$\xymatrix{
 V \ar[d]_p\ar[rr]^f&&\K\\
\mathcal P\ar[rru]_{g} &&
}$$
\end{lemm}

\begin{proof}[Démonstration]
Pour $k\in \N$ on note $\mathcal P_k$ le squelette de dimension $k$ de $\mathcal P$. Commençons par définir une application $\Phi_1:\mathcal P_1\to V$ de la manière suivante :

\begin{enumerate}
	\item Pour tout $i\in \set{1,\ldots,N}$, on pose $\Phi_1(p_i)=v_i$, où, rappelons le, $v_i$ est le centre de la boule $O_i$ du recouvrement $\mathcal O$. 
	\item Lorsque $p_i$ et $p_j$ sont deux sommets de $\mathcal P$ reliés par une arête $\gamma_{i,j}:[0,1]\to P$, on pose, pour tout $t$ dans $[0,1]$ :
$$\Phi_1(\gamma_{i,j}(t))=c_{i,j}(t),$$
où $c_{i,j}:[0,1]\to V$ est un chemin minimisant de $v_i$ à $v_j$.
\end{enumerate}
On peut noter que l'application $$p\circ \Phi_1: \mathcal P_1\to \mathcal P$$ est homotope à $\id_{\mathcal P_1}$ relativement à $\mathcal P_0$. En effet, prenons $p_i$ et $p_j$ deux sommets de $\mathcal P$. Si $v\in \Phi_1([p_i,p_j])$, alors $p(v)\in \text{star}([p_i,p_j])$, où  $\text{star}([p_i,p_j])$ est l'étoile de $[p_i,p_j]$. Il en résulte que l'application
$p\circ \Phi_1:[p_i,p_j]\to \mathcal P$ est homotope à l'identité de $[p_i,p_j]$ relativement à $\set{p_i,p_j}$.

Pour $i$ et $j$ distincts dans $\set{1,\ldots,N}$ et $t$ dans $[0,1]$, posons :
$$g\big(\gamma_{i,j}(t)\big)=f\circ \Phi_1\big(\gamma_{i,j}(t)\big).$$
On définit ainsi une application $g:\mathcal P_1\to \K$ telle que $g=f\circ \Phi_1$. Si maintenant $\Delta=[p_i,p_j,p_k]$ est un simplexe de dimension 2 dans $\mathcal P$, son bord $\partial \Delta$ est un lacet $c$ tel que :
$$\Long(\Phi_1(c))<\sys(V,f,\mathcal G).$$
Il en résulte que $g(\partial \Delta)=f(\Phi_1(\partial \Delta))$ est un lacet contractile dans $\K$, ce qui permet de prolonger l'application $g$ au squelette de dimension 2 de $\mathcal P$. Comme $\K$ est asphérique, il n'y plus d'obstruction à prolonger $g$ à $\mathcal P$ tout entier.

Il reste à démontrer que $g\circ p$ est homotope à $f$. Quitte à raffiner la structure simpliciale de $V$, on peut supposer que les centres des boules $v_1,\ldots,v_N$ sont des sommets de $V$, et que tout simplexe de dimension 1 de $V$ est de longueur inférieure à $\tfrac 13\sys(V,f,\mathcal G)$.
Comme $p\circ \Phi_1: \mathcal P_1\to \mathcal P$ est homotope à $\id_{\mathcal P_1}$ relativement à $\mathcal P_0$, l'application $$g\circ p:\im \Phi_1\to \K$$ est homotope à la restriction de $f$ à $\im \Phi_1$. Il existe alors une homotopie $h_t:\im \Phi_1\to \K$ telle que $h_0=f$ et $h_1=g\circ p$. Chaque $h_t$ se prolonge à $V_1\cup \im \Phi_1$, où $V_1$ est le squelette de dimension $1$ de $V$.  Soit maintenant $\Delta$ un simplexe de dimension 2 de $V$. Son bord $\partial \Delta$ est un lacet dans $V$ de longueur strictement inférieure à $\tfrac 13\sys (V,f,\mathcal G)$, donc $h_t(\partial \Delta)$ est contractile ce qui permet de prolonger $h_t$ en une homotopie de $f$ à $g\circ p$ sur le squelette de dimension 2 de $V$, et il n'y a plus d'obstruction à prolonger $h_t$ en une homotopie de $f$ à $g\circ p$ sur $V$ tout entier. 

Le fait que l'application $g_*:\pi_1(K)\to \pi$ soit un épimorphisme résulte du fait que l'on peut choisir un cycle $(V,f,\mathcal G)$ qui est normalisé.
\end{proof}


\begin{rema}
Le diagramme commutatif du lemme \ref{lemhomotopie} induit alors un diagramme commutatif au niveau des groupes d'homologie, et on a alors :
$$g_*(p_*[V])=h,$$
où $[V]$ est la classe fondamentale de $V$.
\end{rema}

\subsection{Boules admissibles}

On introduit ici la notion de boules admissibles (voir \cite{Gromov-FRM}, théorème 5.3.B) du cycle géométrique $\eps$-régulier $(V,f,\mathcal G)$. Cette notion est aussi utilisée dans  \cite{BBB}.

\begin{defi}
Soit $\alpha>0$. Soient $R_0=\tfrac 1{12} \:\sys(V,f,\mathcal G)$ et $\eps,r,R$ des réels tels que $\eps\le r\le R< R_0$. Pour $v\in V$, on dira que la boule $B(v,R)$ est $(\alpha,r)$-admissible \index{Boules admissibles} lorsque :
\begin{quote}
\begin{enumerate}
	\item $\vol(B(v,5R))\le \alpha\: \vol (B(v,R))$
	\item Pour tout $R'\in[R,R_0]$, $\vol(B(v,5R'))\ge \alpha\: \vol (B(v,R'))$.
\end{enumerate}
\end{quote}
Une boule de rayon $R_0$ sera dite $(\alpha,r)$-admissible lorsque $\vol(B(v,5R_0))\le \alpha\: \vol (B(v,R_0))$.
\end{defi}

Avec le choix fait de $R_0$, le volume d'une boule admissible est supérieure à $A_nR^n$. Mais on peut obtenir une borne inférieure du volume d'une boule $(\alpha,r)$-admissible en fonction de $R_0$ et de $\alpha$. 

\begin{lemm}
\label{lemvolbouleadmissible}
Soit $\alpha>5^n$. Lorsque $B(v,R)$ est une boule $(\alpha,r)$-admissible du cycle géométrique $\eps$-régulier $(V,f,\mathcal G)$, on a :
\begin{equation}
\label{eqvolbouleadmissible}
	\vol B(v,R)\ge C_n(\alpha) R_0^n,
\end{equation}
où $C_n(\alpha)=5^{-m_0(\alpha)n}A_n$ et $m_0(\alpha)=\dfrac{\log \vol(V,\mathcal G)-n\log R_0-\log A_n}{\log \alpha -n\log 5}$
\end{lemm}

\begin{proof}[Démonstration du lemme \ref{lemvolbouleadmissible}]
Le cycle géométrique $(V,f,\mathcal G)$ est $\eps$-régulier, donc $\vol(B(v,R))\ge A_nR^n$. Soit $m$ un entier tel que :
$$5^{-m} R_0\le R \le 5^{-m+1}R_0.$$
On a alors :
\eq
\vol(V,\mathcal G)&\ge& \vol(B(v,5R_0))\ge \alpha\vol(B(v,R_0))\ge \alpha^2 \vol (B(v,5^{-1}R_0))\\
&\ge&\alpha^3 \vol (B(v,5^{-2}R_0))\ge\ldots\ge \alpha^{m} \vol (B(v,5^{-m+1}R_0))\\
&\ge & \alpha^{m} \vol (B(v,R))\\
&\ge & \alpha^m A_n R^n
\fineq
Ainsi $\vol(V,\mathcal G)\ge \alpha^m A_n 5^{-mn}R_0^n$, donc :
$$\log \vol(V,\mathcal G)-n\log R_0-\log A_n \ge m(\log \alpha -n\log 5).$$
Comme $\log \alpha -n\log 5>0$, on obtient : 
$$m\le \dfrac{\log \vol(V,\mathcal G)-n\log R_0-\log A_n}{\log \alpha -n\log 5}=m_0(\alpha)$$
Il en résulte que $\vol (B(v,R))\ge A_n R^n\ge 5^{-mn}A_nR_0^n\ge 5^{-m_0(\alpha)n} A_n  R_0^n$.
\end{proof}

On peut se poser naturellement la question de l'existence de boules admissibles selon la valeur de $\alpha$. Il est clair que pour $\alpha$ grand, il y aura des boules admissibles centrées en tout point de $V$. Le lemme suivant précise un peu cela.

\begin{lemm}
\label{lemexistencebouleadmissible}
Soit $\alpha>5^n$. Pour $\eps$ suffisamment petit, le cycle géométrique $\eps$-régulier $(V,f,\mathcal G)$ admet des boules $(\alpha,r)$-admissibles centrées en chacun de ses points. 
\end{lemm}

\begin{proof}[Démonstration du lemme \ref{lemexistencebouleadmissible}]
On raisonne par l'absurde. Supposons qu'il existe un point $v$ de $V$ tel qu'aucune boule centrée en $v$ ne soit $(\alpha,r)$-admissible. Soit $m$ la partie entière de $\tfrac{\log R_0-\log r}{\log 5}$. On a alors : $5^{-(m+1)}R_0<r\le 5^{-m}R_0$.

Le volume de la boule $B(v,R_0)$ vérifie  :
\eq
\vol (B(v,R_0))&\ge &\alpha ^m\vol(B(v,r))\ge \alpha ^mA_nr^
n\ge5^{-n}\left(5^{-n}\alpha\right)^mA_nR_0^n\\
&\ge&5^{-n}\left(5^{-n}\alpha\right)^{\log_5\frac{R_0}{5r}}A_nR_0^n.
\fineq
Comme $\underset{r\to 0}{\lim}\left(5^{-n}\alpha\right)^{\log_5\frac{R_0}{5r}}=+\infty$, cette dernière inégalité est impossible pour $r$ suffisamment petit. 
\end{proof}

\subsection{Démonstration du théorème \ref{theopolyedre}}
Rappelons que $(V,f,\mathcal G)$ est un cycle géométrique régulier qui représente la classe $h$ dans $h\in H_n(\pi;\Z)$. Choisissons $\alpha>5^n$, de sorte que l'on puisse appliquer les lemmes \ref{lemvolbouleadmissible} et \ref{lemexistencebouleadmissible}. On construit alors un \og{}système maximal\fg{} $B_1, \ldots,B_N$ de boules $(\alpha,r)$-admissibles disjointes  de $V$ de la manière suivante :

\begin{itemize}
	\item On prend un point $v_1$ de $V$ et on note $B_1$ la boule $(\alpha,r)$-admissible de plus grand rayon (que l'on note $R_1$) centrée en $v$;
\item  Supposons alors construites $B_1,\ldots,B_j$. On prend pour $B_{j+1}$ la boule $(\alpha,r)$-admissible de plus grand rayon (noté $R_{j+1}$) ne rencontrant pas les boules $B_1,\ldots,B_j$.
\end{itemize}
Ce processus de construction est correct, puisque le rayon des boules admissibles est minoré par $r$. Notons $v_1,\ldots,v_N$ les centres des boules ainsi construites. On peut noter  que les boules concentriques $2B_1,\ldots,2B_N$, de rayons respectifs $2R_1,\ldots,2R_N$, recouvrent $V$. En effet, raisonnons  par l'absurde et supposons qu'il existe $v\in V$ qui n'appartienne à aucune des boules  $2B_1,\ldots,2B_N$. On a alors, pour tout $i\in \set{1,\ldots,N}$, $\dist(v,v_i)\ge R_N$. Mais il existe une boule admissible de centre $v$ et de rayon inférieur à $R_N$, ce qui contredit la maximalité du système $B_1, \ldots,B_N$.

Désormais $\mathcal P$ désigne le nerf de ce recouvrement et $p:V\to \mathcal P$ l'application canonique induite par une partition de l'unité subordonnée à ce recouvrement. Selon le lemme \ref{lemhomotopie}, il existe $g:\mathcal P\to \K$ tel que, à homotopie près, le diagramme suivant soit commutatif.
$$\xymatrix{
 V \ar[d]_p\ar[rr]^f&&\K\\
\mathcal P\ar[rru]_{g} &&
}$$

Pour $k\in \set{0,\ldots,N}$, notons $N_k$ le nombre de simplexes de dimension $k$ de $\mathcal P$. Pour majorer $N_0$, il n'y a pas de difficulté. Avec les notations du lemme \ref{lemvolbouleadmissible}, on a :
$$\vol(V,\mathcal G)\ge \dsum_{i=1}^N\vol(B_i)\ge N_0C_n(\alpha) R_0^n.$$

Comme $R_0=\dfrac 1{12}\sys(V,f,\mathcal G)$, on obtient :
$$N_0\le \dfrac {12^n}{C_n(\alpha)}\sigma(V,f,\mathcal G).$$ 

Notons que le résultat du lemme  \ref{lemvolbouleadmissible} est encore valable en augmentant $m_0(\alpha)$. On peut donc choisir  :$$m_0(\alpha)=\dfrac{\log \vol(V,\mathcal G)-n\log R_0}{\log \alpha -n\log 5}.$$
Prenons, pour l'instant,  $\alpha$ de sorte que : $\log \alpha=n\log 5+\log \vol(V,\mathcal G) -n\log R_0.$
Il vient ainsi $m_0(\alpha)=1$ et $C_n(\alpha)=5^{-n}A_n$. On en déduit que :
 $$N_0\le 60^n A_n^{-1}\:\sigma(V,f,\mathcal G).$$

On va maintenant majorer le nombre $N_1$ de simplexes de dimension $1$ de $\mathcal N$ : il s'agit du nombre de doubles intersections entre boules du recouvrement.

Fixons $j\in \set{1,\ldots,N}$. On suppose que la boule $2B_j$ rencontre les boules $2B_{j_r}$ pour $r=1,\ldots,s_j$ et $j_r>j$. La boule concentrique $5B_j$ de rayon $5R_j$ contient alors toutes les boules $B_{j_r}$ (et c'est cette propriété qui motive la définition des boules admissibles). En effet, pour $v\in B_{j_r}$ on a  :
\eq
\dist(v,v_i)&\le &\dist(v,v_{j_r})+\dist(v_{j_r},v_j)\\
&< & R_{j_r}+4R_j\le 5R_j
\fineq

Comme la boule $B_j$ est $\alpha$-admissible, il vient :
$$\vol(B_j)\ge \alpha^{-1}\vol (5B_j)\ge \alpha^{-1}\dsum_{r=1}^{s_j}\vol(B_{j_r})$$

Ainsi, en utilisant le lemme (\ref{lemvolbouleadmissible}), on obtient : 
\eq
\vol(V,\mathcal G)&\ge& \dsum_{j=1}^m\vol (B_j)\ge \alpha^{-1}\dsum_{j=1}^m \left(\dsum_{k=1}^{s_j}\vol(B_{j_k})\right)\\
&\ge &\alpha^{-1}\dsum_{j=1}^m s_j C_n(\alpha)R_0^n\ge \alpha^{-1}C_n(\alpha)R_0^n \underbrace{\dsum_{j=1}^m s_j }_{=N_1}.
\fineq

Avec $R_0=\dfrac 1{12}\:\sys(V,f,\mathcal G)$, il vient :
\begin{equation}
\label{eqN1}
	N_1\le \dfrac{12^{n}\alpha}{C_n(\alpha)}\: \sigma(V,f,\mathcal G).
\end{equation}

Avec différents choix de $\alpha$, on peut alors borner $N_1$ de différentes manières en fonction de $\sigma(h)$. 

Prenons encore $m_0(\alpha)=\dfrac{\log \vol(V,\mathcal G)-n\log R_0}{\log \alpha -n\log 5}$ et $\alpha$ de sorte que :
\eq
\log \alpha&=&n\log 5+\sqrt{\log \vol(V,\mathcal G) -n\log R_0}\\
&=&n\log 5+\sqrt{\log \sigma(V,f,\mathcal G) +n\log 12}.
\fineq
Il vient $m_0(\alpha)=\left(\log \sigma(V,f,\mathcal G) +n\log 12\right)^{\frac 12}$. Avec (\ref{eqN1}) on obtient :
$$N_1\le60^n A_n^{-1} \sigma(V,f,\mathcal G)\: \exp \big((1+n\log 5)\sqrt{\log \sigma(V,f,\mathcal G) +n\log 12}\big).$$

Comme $\sigma(V,f,\mathcal G)\le (1+\eps)\sigma(h)$, on obtient, en faisant tendre $\eps$ vers $0$ ($N_1$ étant indépendant de $\eps$) :
 $$N_1\le 60^n A_n^{-1} \sigma(h) \exp\big((1+n\log 5)\sqrt{\log \sigma(h) +n\log 12}\big).$$

Regardons maintenant comment se majore $N_2$. Fixons $j$ dans $\set{1,\ldots,N}$. Pour $j_r>j$ avec $r\in \set{1,\ldots,s_j}$, on suppose que les boules $B_j$, $B_{j_r}$ et $B_{{j_r}_k}$ où $j_{r_k}>j_r$ et $k\in \set{1,\ldots,s_{j_r}}$ ont une intersection non vide.
Comme précédemment, la boule $B(j,5R_j)$ contient toutes les boules $B_{j_r}$, et chaque boule  $5B_{j_r}$ contient toutes les boules $B_{{j_r}_k}$. Il vient alors :
\eq
\vol(V,\mathcal G)&\ge &\dsum_{j=1}^N \vol(B_j)\ge \alpha^{-1}\dsum_{j=1}^N \vol(5B_j)\ge \alpha^{-1}\dsum_{j=1}^N \dsum_{r=1}^{s_j}\vol (B_{j_r})\\
&\ge &\alpha^{-2}\dsum_{j=1}^N \dsum_{r=1}^{s_j}\vol (5B_{j_r})\ge \alpha^{-2}\dsum_{j=1}^N  \dsum_{r=1}^{s_j}\dsum_{k=1}^{s_{j_r}}\vol (B_{{j_r}_k})\\
&\ge & \alpha^{-2}\dsum_{j=1}^N  \dsum_{r=1}^{s_j}\dsum_{k=1}^{s_{j_r}}C_n(\alpha)R_0^n\ge \alpha^{-2}C_n(\alpha)R_0^n\underbrace{\dsum_{j=1}^N  \dsum_{r=1}^{s_j}s_{j_r}}_{=N_2}.
\fineq

Il vient donc $N_2\le \dfrac{12^n\alpha^2}{C_n(\alpha)}\: \sigma(V,f,g)$. De la même manière, on obtient $N_k\le \dfrac{12^n\alpha^k}{C_n(\alpha)}\: \sigma(V,f,g)$, ce qui donne :
$$N_k\le 12^n\alpha^k\ex{n\log 5\:m_0(\alpha)}\: \sigma(V,f,g).$$

Avec la  même valeur de $m_0(\alpha)$ et de $\alpha$, il vient :
$$\alpha^k=5^{kn}\:\exp\big(k\sqrt{\log \vol(V,\mathcal G) -n\log R_0}\big),$$
d'où : $N_k\le \left(12.5^{k}\right)^n A_n^{-1} \sigma(h) \exp\big((k+n\log 5)\sqrt{\log \sigma(h) +n\log 12}\big)$.\\

Cette formule étant encore valable pour $k=0$, le théorème \ref{theopolyedre} est ainsi démontré.

\begin{rema}
On ne peut remplacer $r$ par $\eps$ dans la définition  des boules $(\alpha,r)$-admissibles. En effet, le nombre $N_k$ de simplexes de dimension $k$  de $\mathcal P$ dépend de $\alpha$ et de $r$. Or, pour obtenir le résultat final, on fait tendre $\eps$ vers $0$.
\end{rema}

\subsection{Une majoration des nombres de Betti d'une variété asphérique}
Le théorème \ref{theopolyedre} permet de majorer les nombres de Betti d'une variété fermée asphérique $M$ en fonction de son volume systolique (comparer avec \cite{Gromov-FRM}, page 75). Rappelons qu'un espace topologique $X$ est dit asphérique lorsque pour $i\ge 2$ on a  $\pi_i(X)=0$. \index{Variété asphérique}\\

\begin{theo}[Gromov]
Soit $M$ une variété fermée de dimension $n$, orientable et asphérique. Les nombres de Betti de $M$ vérifient, pour $k\in \set{0,\ldots,n}$ :
$$b_k(M)\le C\sigma(M)\exp(C'\sqrt{\log \sigma(M)}),$$où $C_n$ et $C'_n$ sont deux constantes universelles qui ne dépendent que de la dimension.
\end{theo}
\begin{proof}[Démonstration]
Soit $\pi$ le groupe fondamental de $M$. Puisque $M$ est asphérique, on a $M=\K$. Selon le théorème \ref{theopolyedre}, il existe un polyèdre $\mathcal P$ et une application $\phi:\mathcal P\to M$ tels que :
\begin{itemize}
	\item le nombre $N_k$ de simplexes de dimension $k$ de $\mathcal P$ vérifie :
\begin{equation}
\label{eqnombresimplexes}
		N_k\le  C\sigma(M)\exp\big(C'\sqrt{\log \sigma(M)}\big),
\end{equation}
	où $C_n$ et $C'_n$ sont deux constantes universelles qui ne dépendent que de $n$;
	\item l'application induite en homologie par $\phi$ envoie une certaine classe $h'$ de $H_n(\mathcal P;\Z)$ sur la classe fondamentale $[M]$ de $M$.
\end{itemize}

Fixons $k$ dans $\set{0,\ldots,n}$. On considère le \og{}cup-product\fg{} 
$$H^{k}(M;\R)\times H^{n-k}(M;\R)\to H^{n}(M;\R),$$
ainsi que la \og{}forme bilinéaire\fg{} $\psi_k: H^{k}(M;\R)\times H^{n-k}(M;\R) \to \R$ définie, pour $(\alpha,\beta)$ dans  $H^{k}(M;\R)\times H^{n-k}(M;\R)$, par :
$$\psi_k(\alpha,\beta)=\alpha\cupp\beta([M])\in \R.$$
Rappelons que  $H^k(M;\R)$ est le dual algébrique de $H_k(M;\R)$. Pour $\beta$ dans $H^{n-k}(M;\R)$ considérons $$\psi_k^{\beta}:\alpha \mapsto \psi_k(\alpha,\beta)=\alpha\cupp\beta([M]),$$
qui est une forme linéaire sur $H^k(M;\R)$. On dispose ainsi d'une application linéaire :
$$\varphi_k:H^{n-k}(M;\R)\to H^k(M;\R)^*,$$
définie par $\varphi_k(\beta)=\psi_k^{\beta}$, où $H^k(M;\R)^*$ est le dual algébrique de $H^k(M;\R)$. On peut alors définir le rang de $[M]$ par :
$$\rg_k([M])=\rg(\varphi_k).$$
Ce rang dépend bien sûr de $k$, et on peut définir de la même manière le rang d'une classe homologique quelconque de $H_n(M;\R)$ ou de $H_n(\mathcal P;\R)$. On a alors : 
\begin{equation}
\label{eqrang}
	\rg_k([M])\le \rg_k(h')\le \dim H^k(P;\R)\le N_k.
\end{equation}
Désignons maintenant par $D:H^k(M;\R)\to H_{n-k}(M;\R)$ l'isomorphisme de Poincaré (voir \cite{Hatcher}). Si $\capp$ désigne le \og{}cap-product\fg{} :
$$\capp:H_n(M;\R)\times H_k(M,\R)\to H_{n-k}(M;\R),$$ 
on a $D(\alpha)=[M]\capp\alpha$. Soit  $\trans D:H_{n-k}(M)^* \to H^k(M)^*$ la transposée de $D$. L'application $\trans D$ est un isomorphisme et, pour $\alpha\in H^k(M;\R)$ et $\beta\in H^{n-k}(M;\R)$, on dispose de la relation classique (\cite{Hatcher} page 249) :
$$\alpha\cupp\beta([M])=\beta([M]\capp(\alpha))=\beta\circ D(\alpha)=\trans D(\beta)(\alpha).$$
Ainsi $\varphi_k=\trans D$ : c'est un isomorphisme, donc $\rg(\varphi_k)=b_k(M)$. On peut alors conclure en combinant (\ref{eqnombresimplexes}) et (\ref{eqrang}).
\end{proof}

\begin{rema}
Pour une classe homologique quelconque $h\in H_n(\pi;\R)$, la démonstration ci-dessus permet de dire que :
$$\rg_k(h)\le  C_n\sigma(h)\exp(C'_n\sqrt{\log \sigma(h)}),
$$
où $C_n$ et $C'_n$ sont deux constantes universelles qui ne dépendent que de $n$.
\end{rema}

\section{Entropie volumique et  volume systolique}

\subsection{Introduction}
Le but de ce paragraphe est de présenter comment la technique de régularisation permet d'obtenir des relations liant  l'entropie volumique et le volume systolique. Avant de voir ces relations, il est nécessaire de rappeler quelques notions. Dans la suite, toutes les variétés considérées sont compactes et sans bord. On verra aussi une relation entre l'entropie volumique, le volume systolique et le quotient embolique d'une variété compacte.

\subsection{Entropie volumique et entropie volumique minimale d'une variété}
Passons maintenant aux  notions d'entropie volumique et d'entropie volumique minimale. Soit $(M,g)$  une variété riemannienne  compacte sans bord. On fixe $m_0\in M$ et $\w m_0\in \w M$ tel que $p(\w x_0)=m_0$. On notera  dans la suite $p:\w M \to M$ son revêtement universel riemannien et on identifiera $\pi_1(M,m_0)$ et le groupe des automorphismes du revêtement.

\begin{defi}
\label{Entropievol}
\textit{L'entropie volumique}\index{Entropie volumique} de $(M,g)$ est :
$$h_{\text{vol}}(g)=\underset{R\to +\infty}{\lim}\dfrac{\log(\vol (B(\w m_0,R))}{R},$$
où $B(\w m_0,R))$ désigne la boule de centre $\w m_0$ et de rayon $R$ dans $\w M$.
\end{defi}

Cette limite existe bien, par compacité de $M$, et ne dépend pas du choix du point $\w m_0$ (voir \cite{Manning}).

\begin{defi} Soit $M$ une variété  compacte sans bord de dimension $n\ge 2$. \textit{L'entropie volumique minimale}\index{Entropie volumique minimale} de $M$ est :
 $$\omega(M)=\underset{g}{\inf}\ h_{\text{vol}}(M,g) \vol(M,g)^{\frac{1}{n}},$$
 l'infimum étant pris sur toutes les métriques riemanniennes sur $M$.
\end{defi}

\begin{rema}
\begin{enumerate}
	\item L'entropie volumique minimale est un invariant du type d'homotopie de $M$. Cela est démontré par Babenko dans  \cite{Babenko-AISM}.
	\item Lorsque $M$ admet une métrique localement symétrique $g_0$ de courbure négative alors :
	$$\omega(M)=h_{\text{vol}}(M,g_0)\vol(M,g_0)^{\frac{1}{n}}.$$
	Le cas $n=2$ a été démontré par Katok (voir \cite{Katok2}) et le cas $n=3$ par Besson, Courtois et Gallot (voir \cite{BCG}).
	\item Gromov a relié l'entropie volumique minimale au volume simplicial $\norm{M}$ d'une variété $M$ de dimension $n$, où $\norm{M}=\inf\set{\dsum \mod{r_i}}$, l'infimum étant pris sur tous les cycles réels $\dsum r_i\sigma_i$ qui représentent la classe fondamentale de $M$. Plus précisément, il a montré dans \cite{Gromov-VBC} qu'il existe une constante $C_n$ qui ne dépend que de la dimension de $M$ telle que :
	$$\omega(M)\ge C_n\norm{M}.$$
	\item Dans \cite{Brunnbauer-HI}, Brunnbauer a montré (par exemple) que, pour une variété orientable $M$, de dimension $n$ de groupe fondamental $\pi$, alors $\omega(M)$ ne dépend que de la valeur de $f_*[M]\in H_n(\pi;\Z)$ où $f_*$ est l'application induite en homologie par l'application classifiante $f:M\to K(\pi,1)$. 
	\item Manning a démontré que l'entropie volumique de la variété $(M,g)$ est reliée à l'entropie topologique $h_{top}(M,g)$ du flot géodésique de $(M,g)$ (voir \cite{Manning} et \cite{Katok-Hasselblatt}) par l'inégalité :
	$$h_{vol}(M,g)\le h_{top}(M,g).$$
De plus il y a égalité lorsque $(M,g)$ est à courbure négative.
\end{enumerate}
\end{rema}
\subsection{Entropie volumique relative}
On peut définir différentes entropies volumiques relatives de la manière suivante. Soit $(M,g)$ une variété riemannienne compacte de groupe fondamental $\pi$. Il est possible, pour tout sous groupe distingué $G$ de $\pi$, de définir l'entropie volumique de $M$ associée à $G$ en considérant le revêtement galoisien $\w M\to M$ associé à $G$, au lieu du revêtement universel de $M$, dans la définition \ref{Entropievol} (voir \cite{Babenko-AISM}). Notons $h_{\text{vol}}(M,g,G)$ l'entropie ainsi définie\index{Entropie volumique relative}. En particulier, on a :  
$h_{\text{vol}}(M,g,\pi)=h_{\text{vol}}(M,g)$.\\
\textit{L'entropie volumique minimale associée à} $G$, est alors (par définition) :
 $$\omega_G(M)=\underset{g}{\inf} \:  h_{\text{vol}}(M,g,G),$$ 
 l'infimum étant pris sur toutes les métriques riemanniennes sur $M$ de volume 1.
  
Parmi ces entropies minimales, on retrouve $\omega(M)$. Il y aussi \textit{l'entropie minimale homologique}  $\omega_H(M)$ associée au sous-groupe  $H=\ker \psi$ de $\pi$, où $\psi:\pi_1(M)\to \quot{H_1(M)}{\text{Tor}}$. Babenko a démontré,  dans \cite{Babenko-AISM}, que toutes ces entropies minimales sont des invariants homotopiques de $M$ et que, par exemple, dans le cas où $M$ est orientable, lorsque  l'application classifiante
$f:M\to K(\pi_1(M),1)$ vérifie $f_*[M]=0$, alors $\omega_G(M)=0$ pour tout sous-groupe distingué $G$ d'indice infini dans $\pi_1(M)$.

Soient maintenant $\pi$  un groupe de présentation finie, $h\in H_n(\pi;\Z)$ et $(V,f,g)$ un cycle géométrique représentant $h$. On peut définir (voir \cite{Sabourau-SVME}) de la même manière l'entropie volumique du cycle géométrique $(V,f,g)$ en regardant la croissance du volume des boules dans le revêtement galoisien $\w V\to V$ associé au sous-groupe distingué $\ker f_*$ de $\pi_1(V)$:
$$h_{vol}(V,f,g)=\underset{R\to +\infty}{\lim}\dfrac{\log(\vol (B(\w m_0,R))}{R},$$
où $\w m_0\in\w V$. On a alors le résultat suivant :

\begin{theo}[Sabourau]
\label{theoSabourauIllus1}
Pour tout cycle géométrique $(V,f,g)$ $\eps$-régulier qui représente une classe d'homologie non triviale dans $H_n(\pi;\Z)$, on a :
$$h_{vol}(V,f,g)\le \dfrac{1}{\beta\: \sys(V,f,g)}\log\left(\dfrac{\sigma(V,f,g)}{A_n\alpha^n}\right)$$
où $\alpha\ge \eps$, $\beta>0$, $4\alpha+\beta<\tfrac 12$ et la constante $A_n$ est donnée par le théorème \normalfont{A}.
\end{theo}

On va voir dans la suite la démonstration de ce résultat. Mais auparavant, on a besoin de rappeler quelques résultats sur l'entropie.

\begin{rema}
On peut définir l'entropie volumique minimale de $h$ par :
$$\omega(h)=\underset{(X,f,g)}{\inf} h_{vol}(V,f,g)\vol(V,\mathcal G)^{\frac 1n},$$
où $(V,f,g)$ décrit l'ensemble des cycles géométriques qui représentent $h$.
L'inégalité du théorème \ref{theoSabourauIllus1} donne alors, pour $\alpha>0$ et $\beta>0$ qui vérifient  $4\alpha+\beta<\tfrac 12$ :
$$\omega(h)\le \dfrac{\sigma(h)^{\frac 1n}}{\beta}\log\left(\dfrac{\sigma(h)}{A_n\alpha^n}\right).$$

\end{rema} 

\subsection{Entropie volumique et classes d'homotopie des lacets}

Avant de démontrer le théorème \ref{theoSabourauIllus1}, rappelons une autre définition de l'entropie volumique d'une variété riemannienne compacte, qui est reliée au nombre $P_{m_0}(L)$ de classes d'homotopie de lacets de point base $m_0$ ayant un représentant de longueur inférieure à $L$. Soit  $(M,g)$ une variété riemannienne compacte de groupe fondamental $\pi$. On note $\w \dist$ la distance induite par celle de $M$ sur son revêtement universel riemannien $\w M$. Dans la suite, on désignera par $\Delta$ un domaine fondamental de l'action de $\pi$ sur $\w M$ et on désignera par $\mathcal H(m_0,L)$ l'ensemble des classes d'homotopie de point base $m_0\in M$ qui peuvent être représentées par un lacet de longueur inférieure à $L$. Je rappelle le résultat bien connu suivant (voir pas exemple \cite{Katz-SG}).

\begin{theo}
\label{entropiecomptage}
Soient  $(M,g)$  une variété riemannienne compacte et  $m_0\in M$. On a alors :
$$h_{\text{vol}}(M,g)=\underset{L\to +\infty}{\lim}\dfrac{\log(P_{m_0}(L))}{L}.$$
\end{theo}

On commence par montrer le lemme suivant.

\begin{lemm}
\label{lemmeclassehomotopie}
On a :
$$P_{m_0}(L)=\text{\normalfont{card}}\set{\gamma\in \pi\mid \w \dist(\w m_0,\gamma\cdot \w  m_0)\le L}.$$
\end{lemm}
\begin{proof}[Démonstration du lemme \ref{lemmeclassehomotopie}]
Tout lacet $c$ de point base $m_0$ dans $M$ se relève en un unique chemin $\w c$ dans $\w M$ d'origine $\w m_0$. Si on prend maintenant deux lacets $c_1$ et $c_2$ de point base $m_0$ dans $M$, ces deux lacets sont homotopes si et seulement si les chemins $\w c_1$ et $\w c_2$ (qui ont pour origine $\w m_0$) ont mêmes extrémités (propriété de relèvement des homotopies) i.e. si et seulement s'il existe $\gamma\in \pi$ tel que 
$$\w c_1(1)=\w c_2(1)=\gamma \cdot \w m_0.$$

Soit $H\in \mathcal H(m_0,L)$, représentée par un lacet $c:[0,1]\to M$ de longueur inférieure à $L$. On associe à cette classe l'unique élément $\gamma\in \pi$ tel que $\w c(1)=\gamma \cdot\w m_0$. On a alors :  
$$\w \dist (\w m_0,\gamma\w m_0)\le \text{long}(\w c)=\text{long}(c)=L.$$

On a donc correctement défini une application clairement injective :
$$\phi:\mathcal H(m_0,L)\to \set{\gamma\in \pi\mid \w \dist(\w m_0,\gamma\cdot \w  m_0)\le L}.$$

Il reste à montrer que $\phi$ est surjective. Prenons $\gamma\in \pi$ tel que $\w \dist(\w m_0,\gamma\cdot\w m_0)\le L$.
Soit $\w c$ une géodésique minimisante de $\w m_0$ à $\gamma\cdot\w m_0$. Le lacet $c=p\circ \w c$ (qui est de point base $m_0$) est de longueur inférieure à $L$ et $\gamma$ est l'image de la classe d'homotopie de $c$ par $\phi$.  
\end{proof}

\begin{proof}[Démonstration du théorème \ref{entropiecomptage}]
Soit $D$ le diamètre de $\Delta$ ; c'est aussi le diamètre de tout translaté de $\Delta$ par un élément de $\pi$, puisque $\pi$ agit par isométries sur $\w M$. Posons $N_1(L)=\text{card}\set{\gamma\in \pi\mid \gamma\cdot\Delta\subset \overline B (\w m_0,L)\neq\emptyset }$ où $\overline B (\w m_0,L)$ désigne la boule fermée de centre $\w m_0$ et de rayon $L$. On a :
$$N_1(L)\le \text{\normalfont{card}}\set{\gamma\in \pi\mid \w \dist(\w m_0,\gamma\cdot \w  m_0)\le L}=P_{m_0}(L).$$
Or si $\w m\in B (\w m_0,L-D)$, il existe $\w p\in \Delta$ et $\gamma\in \pi$ tel que $\w m=\gamma\cdot\w p$, donc $\w m\in \gamma\cdot \Delta$ et $\gamma\cdot \Delta\subset\overline B (\w m_0,L)$. Il en résulte que $N_1(L)\:\vol(\Delta,\w g)\ge \vol (B (\w m_0,L-D))$ et ainsi :
$$P_{m_0}(L)\ge \dfrac{\vol (B (\w m_0,L-D))}{\vol(M,g)}.$$
En passant au logarithme naturel, en divisant par $L$ et en passant à la limite inf lorsque $L$ tend vers $+\infty$ :
\begin{equation}
\label{liminf}
\underset{L\to +\infty}{\text{lim\:inf}}\:\dfrac{\log P_{m_0}(L)}{L}\ge  h_{vol}(M,g)
\end{equation}
Posons maintenant $N_2(L)=\text{card}\set{\gamma\in \pi\mid  \gamma\cdot\Delta\subset \overline B (m_0,L+D)}$. On a :
$$\text{\normalfont{card}}\left\{\pi\cdot\w m_0  \cap \overline B(\w m_0,L)\neq \varnothing\right\}\le N_2(L).$$
Mais  $N_2(L)\:\vol(\Delta,\w g)\le \vol(B(\w m_0,L+D))$. Il en résulte que $P_{m_0}(L)\le \dfrac{\vol(B(\w m_0,L+D))}{\vol(M,g)}$, et ainsi :
\begin{equation}
\label{limsup}
	\underset{L\to +\infty}\limsup \dfrac{\log P_{m_0}(L)}{L}\le h_{vol}(M,g)
\end{equation}
Avec (\ref{liminf}) et (\ref{limsup}), on obtient le résultat souhaité.
\end{proof}

\begin{rema}
Le théorème \ref{entropiecomptage} est encore valable pour l'entropie relative $h_{vol}(V,f,g)$ d'un cycle géométrique, la démonstration en est identique.
\end{rema}

\subsection{Démonstration du théorème \ref{theoSabourauIllus1}}

C'est le modèle type d'illustration de la technique de régularisation. Je vais suivre \cite{Sabourau-SVME}. On établit d'abord le résultat suivant, qui nous conduira directement à la preuve du théorème \ref{theoSabourauIllus1} .

\begin{lemm}[Sabourau]
\label{theointermediareentropieSabourau}
Soient $\pi$ un groupe de présentation finie, $(V,f,g)$ un cycle géométrique représentant une classe $h\in H_n(\pi;\Z)$. Pour  $\alpha>0$, soit $(B_1,\ldots,B_{N\alpha})$ un système maximal de boules ouvertes disjointes de $M$ de rayon 
\begin{equation}
	\label{defR}
	R_\alpha\le \alpha\:\text{sys}(V,f,g).
\end{equation}
Lorsque $\beta>0$ vérifie $\alpha+2\beta<\tfrac 12$, on a:
 $$h_{\text{\vol}}(V,f,g)\: \sys (V,f,g)\le \dfrac{\log N_\alpha}{\beta}.$$
\end{lemm}

\begin{proof}[Démonstration]
Notons que les boules concentriques $2B_1,\ldots2B_{N_\alpha}$ de rayon $2R$ recouvrent $V$. Pour $i\in \set{1,\ldots,N_\alpha}$, on note $x_i$ le centre de la boule $B_i$ et on pose, $m_0=x_1$. 

Soit $c:[0,L]\to V$ un lacet paramétré par la longueur d'arc avec $m_0=c(0)$, tel que $f\circ c$ ne soit pas un lacet contractile dans $\K$ . Soit $\beta>0$. On pose :
\begin{equation}
\label{Lbeta}
	L_\beta=\beta\:\sys(V,f,g)
\end{equation}

Supposons que $L_\beta\le L$. Pour $k$ entier naturel tel que $kL_\beta\le L$, on considère les points $m_k=c(kL_\beta)$.

Notons $s$ le plus grand des entiers $k$ tels que $kL_\beta\le L$. Soit $i\in \set{0,\ldots ,s}$. Le point $m_i$ est recouvert : il existe donc $j_i\in \set{1,\ldots ,N_\alpha}$ tel que $m_i\in B(x_{j_i},2R)$. On pose $p_i=x_{j_i}$, et on a donc $\dist(m_i,p_i)<2R$.

\begin{figure}[h]
\begin{center}
\includegraphics[width=0.5\linewidth]{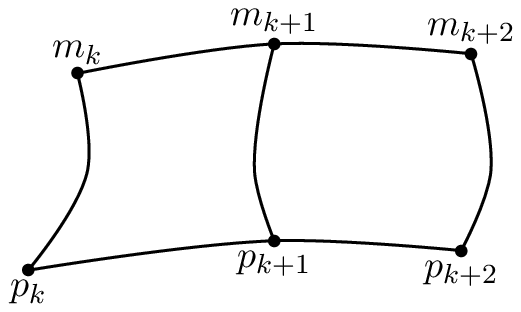}
\end{center}
\end{figure}

Pour $m$ et $q$ points de $M$, on notera $[m,q]$ un plus court chemin de $m$ à $q$. Pour chaque $k\in \left\{1,\ldots,s-1\right\}$, on note $c_k$ la restriction de $c$ au segment $[kL_\beta,(k+1)L_\beta]$, et on considère le lacet $\alpha_k$  défini par :
$$\alpha_k=c_k\ast[m_{k+1}, p_{k+1}]\ast [p_{k+1},p_k]\ast [p_{k},m_k],$$
où $\ast$ désigne la concaténation des chemins. Bien sûr, il peut y avoir plusieurs plus courts chemins de $m_k$ à $p_k$ : on choisit le même chemin pour la construction des lacets $\alpha_k$ et $\alpha_{k-1}$. Il vient alors :
\begin{eqnarray*}
	\Long(\alpha_k)&=& \Long(c_k)+\dist(m_{k+1},p_{k+1})+\dist(p_{k+1},p_k)+\dist(m_k,p_k)\\
	&\le &L_\beta+4R_\alpha+\dist(p_k,p_{k+1})\\
	&\le &2(L_\beta+4R_\alpha)\\
		&\le &2(\beta+4\alpha)\sys(V,f,g)) 
\end{eqnarray*} 

Ainsi, lorsque  $2(\beta+4\alpha)<1$, l'image du lacet  $\alpha_k$ par l'application $f:V\to \K$ est contractile. Il en va de même pour les lacets $$\alpha_0=c_0\ast[m_1,p_1]\ast[p_1,m_0]\et \alpha_s=c_s\ast[m_s,p_s]\ast[p_s,m_0].$$

Considérons alors le lacet  
\begin{equation}
\label{constructionlacet}
	c'=[m_0,p_1]\ast[p_1,p_2]\ast\cdots\ast[p_s,m_0],
\end{equation}
D'après ce qui précède, le lacet $f\circ c'$ est homotope au lacet  $f\circ c$. De plus, deux lacets  $c_1$ et $c_2$ de point base $m_0$ qui représentent  des éléments distincts de $\quot{\pi_1(V)}{\ker f_*}$ permettent de construire, par (\ref{constructionlacet}), deux lacets  $c_1'$ et $c_2'$ qui représentent encore des éléments distincts de $\quot{\pi_1(V)}{\ker f_*}$.

Or il y a $N_\alpha^s$  lacets distincts construits à partir de la formule (\ref{constructionlacet}) (en choisissant, bien entendu, le même chemin minimisant entre les centres des boules du recouvrement).\\     

Ainsi : $P_{m_0}(L)\le N_\alpha^s$. Comme $s\le \dfrac{L}{\beta\: \sys(M,g)}$, il vient alors :$$\log P_{m_0}(L)\le s\log N_\alpha\le \dfrac{L}{\beta\: \sys(V,f,g)} \: \log N_\alpha,$$
ce qui prouve le lemme \ref{theointermediareentropieSabourau}.
\end{proof} 

De ce lemme, découle immédiatement la preuve du théorème \ref{theoSabourauIllus1}. Prenons $R_\alpha=\alpha \:\sys(V,f,g)$. Supposons que $(V,f,g)$ soit un cycle $\eps$-régulier. Puisque $\alpha<\tfrac 12$, le théorème \ref{theo-reg} page \pageref{theo-reg} s'applique : pour tout $i\in \set{1,\ldots,N_\alpha}$ on a $\vol (B_i)\ge A_nrR_\alpha^n$ pour tout $i\in \set{1,\ldots,N_\alpha}$. Ainsi
 $$\vol(V,f,g)\ge N_\alpha A_nrR_\alpha^n,$$ 
et avec l'inégalité du lemme \ref{theointermediareentropieSabourau}, on obtient 
$$h_{\text{\vol}}(V,f,g)\ \sys (V,f,g)\le \dfrac{1}{\beta}\:\log \dfrac{\sigma(V,f,g)}{A_n\alpha^n},$$
ce qui prouve le théorème \ref{theoSabourauIllus1}.


\subsection{Une inégalité liant le volume systolique, l'entropie volumique et le quotient embolique}
Le but de ce paragraphe est de présenter une inégalité
reliant l'entropie volumique minimale, le volume systolique et le quotient embolique
d'une variété compacte. Avant d'énoncer cette inégalité, il est nécessaire
de rappeler  quelques notions. 

Soit $(M,g)$ une variété riemannienne de dimension $n\ge 1$. On considère le rapport\index{Quotient embolique} :
$$\emb(M,g)=\dfrac{\vol(M,g)}{\inj(M,g)^n},$$
où $\inj(M,g)$ désigne le rayon d'injectivité de la variété $(M,g)$, que l'on appellera quotient embolique de $(M,g)$.

Berger a démontré dans \cite{Berger-Emb} l'existence d'une constante $C_n>0$ qui ne dépend que de la dimension de $M$ telle que :
\begin{equation}
\label{embolie}
	\emb(M,g)\ge C_n
\end{equation}
Cela autorise la définition suivante.

\begin{defi}
Soit $M$ une variété (compacte sans bord) de dimension $n$. La \textit{constante embolique}\index{Constante embolique} de $M$ est :
$$\emb(M)=\underset{g}{\inf}\:\emb(M,g),$$
l'infimum étant pris sur toutes les métriques riemanniennes $g$ sur $M$. 
\end{defi}
\begin{rema}
La constante $C_n$ de l'inégalité (\ref{embolie}) est explicite, il s'agit de $\dfrac{\omega_n}{\pi^n}$, où $\omega_n$ désigne le volume de la sphère unité $S^n$ de $\R^{n+1}$. De plus, Berger a démontré qu'il y a égalité dans (\ref{embolie}) si et seulement si $(M,g)$ est la  sphère de dimension $n$ munie de sa métrique canonique $g$.
\end{rema}


On aura besoin dans la suite du résultat suivant, démontré dans \cite{Croke-II}, qui met en jeu le rayon d'injectivité d'une variété compacte (voir \cite{GHL}).

\begin{theo}
\label{InégalitéCroke}
Soit $(M,g)$ est une variété riemannienne compacte sans bord de dimension $n$. Il existe une constante $c_n>0$, qui ne dépend que de la dimension de $M$, telle que les boules $B(R)$ de $M$ de rayon $R\in[0,\tfrac12\:\inj(M,g)]$ vérifient :
\begin{equation}
\label{Croke}
	\vol(B(R))\ge c_nR^n.
\end{equation}
\end{theo}

\begin{rema}
\begin{enumerate}
	\item L'énoncé de ce résultat est à comparé avec le théorème \ref{theo-reg} page \pageref{theo-reg}. Les techniques mises en jeu dans la démonstration utilise bien entendu la courbure, qui n'intervient pas en géométrie systolique.
\item On dispose d'une expression explicite de la constante $c_n$. En désignant par $\omega_n$ le volume de la sphère unité $S^n$ dans $\R^{n+1}$ muni de la métrique canonique, on a 
$$c_n=\dfrac {2^{n-1}\omega_{n-1}^n}{\omega_n^{n-1}n^n}.$$
\end{enumerate}
\end{rema}

On peut maintenant énoncer un résultat qui relie le volume systolique, l'entropie volumique et le quotient embolique d'une variété  compacte.

\begin{theo}
\label{theomoa}
Soit $(M,g)$ une variété riemannienne compacte non simplement connexe.  Pour $\alpha>0$ et $\beta>0$ tels que $\alpha+2\beta<\tfrac 12$, on a :
 $$\omega(M)\le \dfrac {\sigma(M,g)^{\frac 1n}}\beta\:\log \dfrac{\emb(M,g)}{c_n\alpha^n},$$
où $\sigma(M,g)$ est le volume systolique de $(M,g)$ et $c_n$ est la constante définie par le théorème \ref{InégalitéCroke}.
\end{theo}

La preuve du théorème \ref{theomoa} est calquée sur celle du théorème \ref{theoSabourauIllus1}. On a besoin au préalable du lemme suivant (comparer au lemme \ref{theointermediareentropieSabourau}).
 
\begin{lemm}
\label{theointermediareentropiemoa}
Soient $(M,g)$ une variété riemannienne compacte non simplement connexe, $\alpha>0$. Soit $(B_1,\ldots,B_{N\alpha})$ un système maximal de boules ouvertes disjointes de $M$ de rayon 
\begin{equation}
	\label{defR2}
	R_\alpha=\alpha\:\text{inj}(M,g).
\end{equation}
Lorsque $\beta>0$ vérifie $\alpha+2\beta<\tfrac 12$, on a:
 $$h_{\text{\vol}}(M,g)\ \sys(M,g)\le \dfrac{\log N_\alpha}{\beta}.$$
\end{lemm}

\begin{proof}[Démonstration du lemme \ref{theointermediareentropiemoa}] 
La preuve suit pas à pas celle du lemme \ref{theointermediareentropieSabourau}. On regarde seulement la contractibilité des lacets dans $M$. Comme tout lacet de longueur $2\:\inj(M,g)$ est contenu dans une boule de rayon $\inj(M,g)$ qui est contractile, car difféomorphe à une boule euclidienne, on a $2\:\inj(M,g)\le \sys(M,g)$. 
On obtient, en gardant les notations  de la démonstration du lemme \ref{theointermediareentropieSabourau}, $\Long(\alpha_k)\le 2(\beta+2\alpha)\:\sys(M,g)$. 
\end{proof}

\begin{proof}[Démonstration du théorème \ref{theomoa}]
Soit $g$ une métrique riemannienne sur $M$. Le théorème \ref{InégalitéCroke} permet de majorer $N_\alpha$. En effet, on a :
$$\vol(M,g)\ge \dsum_{i=1}^{N_\alpha}\vol (B_i)\ge N_\alpha c_nR_\alpha^n.$$
Ainsi $N_\alpha\le \dfrac{\emb(M,g)}{\alpha^nc_n}$, donc, avec le lemme \ref{theointermediareentropiemoa}, il  vient :
$$h_{\text{\vol}}(M,g)\ \sys (M,g)\le \dfrac 1\beta\log \dfrac{\emb(M,g)}{\alpha^nc_n}.$$
En multipliant chacun des membres de cette inégalité par $\vol(M,g)^{\frac 1n}$, on obtient l'inégalité du théorème \ref{theomoa}, en passant à la borne inférieure dans l'inégalité de gauche.
\end{proof}

\begin{rema}
Avec la même technique, Sabourau a prouvé (voir\cite{Sabourau-SVME}), que pour $\alpha+4\beta<\tfrac12$, on a :
$$\omega(M)\le \dfrac{\emb(M)^{\frac 1n}}{\beta}\:\log \dfrac{\emb(M)}{c_n\alpha^n},$$
en prenant $L_\beta=\beta\:\inj(M,g)$ dans (\ref{Lbeta}).
\end{rema}


\bibliographystyle{amsalpha}
\bibliography{biblio_actes_TSG_v5}
\end{document}